\documentclass[preprint]{amsart}

\usepackage{amssymb}
\usepackage{amsthm}
\usepackage{mathrsfs}
\usepackage{framed}
\usepackage{amsmath}
\usepackage{graphicx}
\usepackage{epsfig}
\usepackage[T1]{fontenc}
\usepackage{hyperref}
\newtheorem{theorem}{Theorem}[section]

\theoremstyle{definition}
\newtheorem{definition}[theorem]{Definition}

\theoremstyle{remark}
\newtheorem{remark}[theorem]{Remark}
\theoremstyle{proposition}
\newtheorem{proposition}[theorem]{Proposition}

\theoremstyle{corollary}
\newtheorem{corollary}[theorem]{Corollary}
\newcommand{\Rn}{\mathbb{R}^n}
\newcommand{\R}{\mathbb{R}}

\newcommand{\Cn}{\mathbb{C}^n}
\newcommand{\C}{\mathbb{C}}

\newcommand{\Z}{\mathbb{Z}}

\newcommand{\abs}[1]{\left\lvert#1\right\rvert}

\DeclareMathOperator{\real}{Re}
\DeclareMathOperator{\imag}{Im}
\newcommand{\im}{\operatorname{Im}}
\newcommand{\re}{\operatorname{Re}}

\usepackage{mathtools}
\mathtoolsset{showonlyrefs}

\begin{document}
\title[On Rad\'o's theorem for polyanalytic functions]{On Rad\'o's theorem for polyanalytic functions}

\author[Abtin~Daghighi]{Abtin~Daghighi}
\thanks{*Corresponding author: abtindaghighi@gmail.com}


\begin{abstract}
We prove versions of Rad\'o's theorem for polyanalytic functions in one variable and also on simply connected $\mathbb{C}$-convex domains in $\mathbb{C}^n$.
Let $\Omega\subset \mathbb{C}$ be a bounded, simply connected domain and let $q\in \mathbb{Z}_+.$ 
Suppose at least one of the following conditions holds true:
(i) $g\in C^{q}(\Omega).$
(ii) $g\in C^\kappa(\Omega),$ for $\kappa=\min\{1,q-1\},$ such that $g$ is $q$-analytic on $\Omega\setminus g^{-1}(0)$ and such that $\re g$ ($\im g$ respectively) is a solutions to the $p'$-Laplace equation ($p''$-Laplace equation respectively) on $\Omega\setminus g^{-1}(0)$, for some $p',p''>1$.
Then $g$ agrees (Lebesgue) a.e.\ with a function that is $q$-analytic on $\Omega.$ 
\\
In the process we give a simple proof of the fact that 
if $f\in C^q(\Omega)$
is $q$-analytic on $\Omega\setminus f^{-1}(0)$
then $f$ is $q$-analytic on $\Omega.$
The extensions of the results to several complex variables are straightforward using known techniques.
\end{abstract}

\maketitle

\section{Introduction}

Rad\'o's theorem states that a continuous function on an open subset of $\Cn$ that is holomorphic off its zero set extends to a holomorphic function on the given open set. For the one-dimensional result see Rad\'o \cite{t1}, and for a generalization to several variables, see e.g.~Cartan \cite{cartan}.

\begin{definition}
Let $\Omega\subset \C$ be an open subset. A function $f$ on $\Omega$ is called \emph{polyharmonic of order $q$} if
$\Delta^q f=0$ on $\Omega$, where $\Delta$ denotes the Laplace operator.
\end{definition} 

\begin{definition}
	Let $\Omega\subseteq\Rn$ be an open subset. For a fixed $p>1,$
	the {\em $p$-Laplace} operator of a real-valued function $u$ on $\Omega$ is defined as
	\begin{equation}
	\Delta_p:=\mbox{div}(\abs{\nabla u}^{p-2}\nabla u)
	\end{equation}
	The operator can also be defined for $p=1$ (it is then the negative of the so-called mean curvature operator) and $p=\infty$ but we shall not concern ourselves with such cases.
\end{definition}

\begin{remark}\label{regrem}
	Note the subtle similarity between the notation for the $p$-Laplace operator 
	\begin{equation}
	\Delta_p=\mbox{div}(\abs{\nabla u}^{p-2}\nabla u)
	\end{equation}
	and that of the $p$:th power of the Laplace operator $\Delta^p$.
	We have that $\Delta_{2} =\Delta.$ 
	More generally, we have
	\begin{equation}
	\Delta_p u= \abs{\nabla u}^{p-4}\left( \abs{\nabla u}^{2}\Delta u +(p-2)
	\sum_{i,j=1}^n\partial_{x_i} u \cdot \partial_{x_j} u \cdot \partial_{x_i}\partial_{x_j} u\right)
	\end{equation}
	Note that $\Delta_p$ is {\em quasilinear}.
	At least they both share the property of being elliptic operators. In the case of $\Delta^p$ this is a direct consequence of the fact that $\Delta$ is a elliptic operator and therefore any finite power is also, in particular the elliptic regularity theorem
	applies to $\Delta^p$ and to $\Delta_p$, and implies that any real-valued 
	distribution solution $u$ to $\Delta^p u=0$ (or to $\Delta_p$) on a domain $\Omega\subset\R^n$ is Lebesgue a.e.\ equal to a $C^\infty$-smooth solution
	$\tilde{u}$ to $\Delta^p \tilde{u}=0$ (or to $\Delta_p\tilde{u}=0$) on $\Omega.$
\end{remark}
Kilpel\"ainen \cite{kilpelainen} proved the following.
\begin{theorem}\label{kipelthm}
	If $\omega\subset\R^2$ is a domain and if $u\in C^1(\Omega)$ satisfies the $p$-Laplace equation
	$\mbox{div}(\abs{\nabla}^{p-2}\nabla u)=0$ on $\Omega\setminus u^{-1}(0)$ then $u$ is a solution to the $p$-Laplacian on $\Omega.$
\end{theorem}
We mention that, more recently, Tarkhanov \& Ly \cite{tarkhanov} proved the following related result in higher dimension.
\begin{theorem}\label{tarkhanovthm}
	Let $\Omega\subseteq\Rn$ be an open subset. If $u\in C^{1,\frac{1}{p-1}}(\Omega)$ such that
	$\mbox{div}(\abs{\nabla}^{p-2}\nabla u)=0$ on $\Omega\setminus u^{-1}(0)$ then this holds true on all of $\Omega.$
\end{theorem}

We shall use the result of Kilpel\"ainen \cite{kilpelainen} in order to prove a natural version of Rad\'o's theorem for polyanalytic functions.
Avanissian \& Traor\'e~\cite{avan1},~\cite{avan2} introduced the
following definition of polyanalytic functions of order $\alpha\in \Z_+^n$ in
several variables.
\begin{definition}\label{avandef}
Let $\Omega\subset \Cn$ be a domain, let $\alpha\in \Z_+^n$ and let $z=x+iy$ denote
holomorphic coordinates in $\Cn$.  A function $f$ on $\Omega$ is called \emph{polyanalytic of order $\alpha$} if in a neighborhood of
every point of $\Omega$, $\left(\frac{\partial}{\partial
\bar{z}_j}\right)^{\alpha_j}\!f(z)=0,1\leq j\leq n$. 
\end{definition}

\begin{definition}
Let $\Omega\subset\Cn$ be an open subset and let $(z_1,\ldots,z_n)$ denote holomorphic coordinates for $\Cn.$
A function $f$, on $\Omega,$ is said to be {\em separately $C^{k}$-smooth with respect to the $z_j$-variable}, if
for any fixed $(c_1,\ldots,c_{n-1})\in \C^{n-1},$ chosen such that
the function 
\[ z_j \mapsto f(c_1,\ldots,c_{j-1},z_j,c_j,\ldots,c_{n-1}), \]
is well-defined (i.e.\ such that $(c_1,\ldots,c_{j-1},z_j,c_j,\ldots,c_{n-1})$ belongs to the domain of $f$) is $C^{k}$-smooth with respect to $\re z_j, \im z_j$.
For $\alpha\in \Z_+^n$ we say that $f$ is separately $\alpha$-smooth if $f$ is separately
$C^{\alpha_j}$-smooth with respect to $z_j$ for each $1\leq j\leq n$.
\end{definition}

We shall need the following result.
\begin{theorem}(See \cite[Theorem 1.3, p.\,264]{avan2})\label{hartog1}
Let $\Omega\subset\Cn$ be a domain and let $z=(z_1,\ldots,z_n),$
denote holomorphic coordinates in $\Cn$ with $\real z=:x, \imag z=y$. Let
$f$ be a function which, for each $j$, is polyanalytic of order
$\alpha_j$ in the variable $z_j=x_j+iy_j$ (in such case we shall
simply say that $f$ is separately polyanalytic of order
$\alpha$). Then $f$ is jointly smooth with respect to
$(x,y)$ on $\Omega$ and furthermore is polyanalytic of order
$\alpha=(\alpha_1,\ldots,\alpha_n)$ in the sense of Definition~\ref{avandef}.
\end{theorem}

\section{Statement and proof of the result}

Let us make the following first observation.
\begin{proposition}\label{halvmainlemma}
Let $\Omega\subset\C$ be a simply connected domain, let $q\in \Z_+$ and let $f\in C^q(\Omega)$
be a $q$-analytic function on $\Omega\setminus f^{-1}(0).$
Then $f$ is $q$-analytic on $\Omega.$
\end{proposition}
\begin{proof}
If $f\equiv 0$ then we are done, so assume $f\not\equiv 0.$
Since $f$ is $C^q$-smooth the function
$\partial_{\bar{z}}^q f$ is continuous.
By assumption $\partial_{\bar{z}}^q f=0$ on $\Omega\setminus f^{-1}(0).$
Set $Z:=(f^{-1}(0))^\circ$ ($^\circ$ denoting the interior) and $X:=\{f\neq 0\}\cup Z.$ Now $f|_Z$ clearly satisfies $\partial_{\bar{z}}^q f=0$.
Let $p\in \partial X.$ If $p$ is an isolated zero of $f$, then
by continuity we have $\partial_{\bar{z}}^q f(p)=0.$ Suppose $p$ is a non-isolated zero.
We have for each sufficiently large $j\in \Z_+$
that $\{\abs{z-p}<1/j\}\cap X\neq \emptyset$.
This implies that there exists a sequence $\{z_j\}_{j\in \Z_+}$
of points $z_j\in X$ such that $z_j\to p$ as $j\to \infty.$
By continuity we have 
\begin{equation}\partial_{\bar{z}}^q f(p)=\lim_{j\to \infty} \partial_{\bar{z}}^q f(z_j) =0\end{equation}
This completes the proof.
\end{proof}

\begin{theorem}\label{halvmain}%
	Let $\Omega\subset \C$ be a bounded, simply connected domain, let $q\in \Z_+$
	and let $f$ be a function $q$-analytic on $\Omega\setminus f^{-1}(0)$. 
	Suppose at least one of the following conditions holds true:\\
	(i) $f\in C^\kappa(\Omega),$ for $\kappa=\min\{1,q-1\},$ and $\re f$ ($\im f$ respectively) is a solutions to the $p'$-Laplace equation ($p''$-Laplace equation respectively) on $\Omega\setminus f^{-1}(0)$, for some $p',p''>1$.\\
	(ii) $f\in C^{q}(\Omega).$
	\\
	Then $f$ agrees (Lebesgue) a.e.\ with a function that is $q$-analytic on $\Omega.$ 
\end{theorem}
\begin{proof}
	The case (ii) follows from Proposition \ref{halvmainlemma}. So suppose (i) holds true.
	If $q=1$ the theorem is well-known and due to Rad\'o \cite{t1}, so assume $q\geq 2.$ 
	Let $f=u+iv$ where $u=\re f,$ $v=\im f.$ 
	Now $f^{-1}(0)= u^{-1}(0) \cap v^{-1}(0)$, whence $u$ (and $v$ respectively) is
	a solution to the $p'$-Laplace equation ($p''$-Laplace equation respectively)
	on $\Omega\setminus u^{-1}(0)$ ($\Omega\setminus v^{-1}(0)$ respectively).
	If $f\in C^{\kappa}(\Omega)$ and $q\geq 2$ then $u$ and $v$ respectively are at least $C^1$-smooth thus satisfy the conditions of 
	Theorem \ref{kipelthm}. Hence
	it follows that $u$ ($v$ respectively) are 
	solutions to the $p'$-Laplace equation ($p''$-Laplace equation respectively) on all of $\Omega$.
	By Remark \ref{regrem} (in particular Elliptic regularity) it follows that $u$ and $v$ respectively agree (Lebesgue) a.e.\ on $\Omega$ with
	$C^\infty$-smooth functions $\tilde{u}$ and $\tilde{v}$ respectively.
	This implies that the function $\tilde{f}:=\tilde{u}+i\tilde{v}$ is $C^\infty$-smooth on $\Omega$
	and agrees (Lebesgue) a.e.\ on $\Omega$ with $f.$
	Suppose there exists a point $p_0\in \Omega$ 
	such that $\partial_{\bar{z}}^q \tilde{f}(p_0)\neq 0.$ Set $Z:=(f^{-1}(0))^\circ$ and $X:=\{f\neq 0\}\cup Z.$ By continuity there exists an open neighborhood $U_{p_0}$ of $p_0$ in $\Omega$
	such that $\partial_{\bar{z}}^q \tilde{f}\neq 0$ on the open subset $U_{p_0}\cap X.$ By the definition of $\tilde{f}$ there exists a set $E$ of zero measure such that
	on $V_{p_0}:=(X\cap U_{p_0})\setminus E$ we have that $\partial_{\bar{z}}^q f$ exists (since $X$ contains no point of $f^{-1}(0)\setminus Z$) and satisfies
	$0=\partial_{\bar{z}}^q f=\partial_{\bar{z}}^q \tilde{f}$ on $V_{p_0},$ which could only happen if
	$V_{p_0}$ is empty which is impossible since $E$ cannot possess interior points.
	We conclude that
	$\partial_{\bar{z}}^q \tilde{f}=0$ on $\Omega.$ This completes the proof.
\end{proof}

\begin{theorem}[Rad\'o's theorem for polyanalytic functions in several complex variables]
	Let $\Omega\subset \Cn$ be a bounded $\mathbb{C}$-convex domain. Let $\alpha\in \Z_+^n$. 
	Suppose $f$ is $\alpha$-analytic on $\Omega\setminus f^{-1}(0)$ such that one of the following conditions hold true:\\
	(i) For each $j=1,\ldots,n$, the function $f$ is separately 
	$C^{\kappa_j}$-smooth with respect to $z_j$ (i.e.\ for each fixed value of the remaining variables $z_k,$ $k\neq j$,
	$f$ becomes a $C^{\kappa_j}$-smooth function of $z_j$), $\kappa_j=\min\{1,\alpha_j-1\}$ and $\re f$ ($\im f$ respectively) are
	solutions to the $p'$-Laplace equation ($p''$-Laplace equation respectively) for some $p',p''>1.$
	\\
	(ii) For each $j=1,\ldots,n$, the function $f$ is separately 
	$C^{\alpha_j}$-smooth with respect to $z_j$.\\
	Then $f$ agrees (Lebesgue) a.e.\ with a function that is $\alpha$-analytic on~$\Omega$. 
\end{theorem}
\begin{proof}
	Denote for a fixed $c\in \C^{n-1}$, $\Omega_{c,k}:=\{ z\in \Omega
	:z_j=c_j,j<k, z_j=c_{j-1},j>k \}$. Since $\Omega$ is $\C$-convex, $\Omega_{c,k}$ is simply connected. Consider the function
	$f_c(z_k):=$ $f(c_1,\ldots,c_{k-1},z_k,c_{k},\ldots,c_{n-1})$.
	Clearly, $f_c$ is $\alpha_k$-analytic on $\Omega_{c,k}\setminus
	f^{-1}(0)$ for any $c\in \C^{n-1}.$ Since $f^{-1}_c(0)\subseteq
	f^{-1}(0)$, Theorem~\ref{halvmain} applies to $f_c$ meaning that
	$f$ agrees a.e.\ with a function $\tilde{f}$ that is {\em separately} polyanalytic of order $\alpha_j$ in the
	variable $z_j, 1\leq j\leq n$.  
	By Theorem~\ref{hartog1} the function $\tilde{f}$ must be polyanalytic of order $\alpha$ 
	on $\Omega$.
	This completes the proof.
\end{proof}

\begin{corollary}
	Let $\Omega\subset \C$ be a bounded $\C$-convex domain and let $\alpha\in \Z_+^n$. 
	Suppose $f$ is separately $C^{\alpha_j}$-smooth with respect to $z_j,$ $j=1,\ldots,n.$
	If $f$ is $\alpha$-analytic on $\Omega\setminus f^{-1}(0),$
	then $f$ agrees (Lebesgue) a.e.\ with a function that is $\alpha$-analytic on~$\Omega$. 
\end{corollary}

\bibliographystyle{amsplain}

\end{document}